\documentclass{birkjour}
 \newtheorem{thm}{Theorem}[section]
 
 \newtheorem{lem}[thm]{Lemma}
 
 \theoremstyle{definition}
 \newtheorem{defn}[thm]{Definition}
 \theoremstyle{remark}
 \newtheorem{rem}[thm]{Remark}

 \usepackage[mathscr]{eucal}
 \usepackage[OT1,T1,TS1,T2A]{fontenc}
\usepackage[cp1251]{inputenc}
\usepackage[russian,english]{babel}
 \numberwithin{equation}{section}
\newcommand{\ccomma}{\mathpunct{\raisebox{0.5ex}{,}}}
\newcommand{\sigmal}{\text{\LARGE\(\sigma\)}}
\newcommand{\taul}{\text{\LARGE\(\tau\)}}
\newcommand{\sigmas}{\text{\small\(\sigma\)}}
\newcommand{\taus}{\text{\small\(\tau\)}}

\newcommand{\taup}{\text{\Large\(\tau\)}}

\DeclareMathOperator{\tr}{trace}

\begin{document}

%
%
%
%
%

\title[On the BMV conjecture for \(2\times 2\) matrices ]
 {On the BMV conjecture for \(\boldsymbol{2\times 2}\) matrices\\
 and the exponential convexity\\
  of the function \(\boldsymbol{\cosh(\sqrt{at^2+b})}\)}
\author[V. Katsnelson]{Victor Katsnelson}
\address{%
Department of Mathematics\\
The Weizmann Institute\\
76100, Rehovot\\
Israel}
\email{victor.katsnelson@weizmann.ac.il; victorkatsnelson@gmail.com}
\date{October 1, 2014}
\begin{abstract}
The BMV conjecture states that for \(n\times n\) Hermitian matrices \(A\) and \(B\)
the function \(f_{A,B}(t)=\tr e^{tA+B}\) is exponentially convex.
Recently the BMV conjecture was proved by Herbert Stahl. The proof of Herbert Stahl is based
on ingenious considerations related to Riemann surfaces of algebraic functions.
In the present paper we give a purely "matrix" proof of the BMV conjecture for
\(2\times2\) matrices. This proof is based on the Lie product formula
for the exponential of the sum of two matrices.
The proof also uses the commutation relations for the Pauli matrices and does not
use anything else.
\end{abstract}
\maketitle
 \section{Herbert Stahl's Theorem.}
In the paper \cite{BMV} a conjecture was
formulated which now is commonly known as the BMV conjecture:\\[1.0ex]
\textbf{The BMV Conjecture.} Let \(A\) and \(B\) be Hermitian matrices of size
\(n\times{}n\).
Then the function
\begin{equation}
\label{TrF}
f_{A,B}(t)=
\textup{trace}\,\{\exp[tA+B]\}
\end{equation}
of the variable \(t\) is representable as a bilateral Laplace transform of a \textsf{non-negative} measure
\(d\sigma_{A,B}(\lambda)\) compactly supported on the real axis:
\begin{equation}
\label{LaR}
f_{A,B}(t)=\!\!\int\limits_{\lambda\in(-\infty,\infty)}\!\!\exp(t\lambda)\,d\sigma_{A,B}(\lambda), \ \ \forall \,t\in(-\infty,\infty).
\end{equation}

\begin{defn}
\label{ETF }
Let \(A,B\) be a pair of square  matrices of the same size \(n\times n\). The function \(f_{A,B}(t)\) of the variable \(t\in\mathbb{R}\) defined by
\eqref{TrF} is said to be \emph{the trace-exponential function generated
by the pair \(A,\,B\).}
\end{defn}

{\ }\\[3.0ex]
Let us note that the function \(f_{A,B}(t)\), considered for \(t\in\mathbb{C}\), is an entire function of exponential type. The indicator diagram of the function \(f_{A,B}\) is
the closed interval \([\lambda_{\min},\lambda_{\max}]\), where \(\lambda_{\min}\) and
\(\lambda_{\max}\) are the least and the greatest eigenvalues of the matrix \(A\) respectively.
Thus if the function \(f_{A,B}(t)\) is representable in the form \eqref{LaR} with a non-negative
measure \(d\sigma_{A,B}(\lambda)\), then \(d\sigma_{A,B}(\lambda)\) is actually supported on the interval
\([\lambda_{\min},\lambda_{\max}]\) and the representation
\begin{equation}
\label{LaRm}
f_{A,B}(t)=
\hspace*{-8.0pt}\int\limits_{\lambda\in[\lambda_{\min},\lambda_{\max}]}\hspace*{-10.0pt}
\exp(t\lambda)\,\,d\sigma_{A,B}(\lambda), \ \ \forall \,t\in\mathbb{C},
\end{equation}
holds for every \(t\in\mathbb{C}\).

The representability of the function \(f_{A,B}(t)\), \eqref{TrF}, in the form \eqref{LaRm}
with a non-negative \(d\sigma_{A,B}\) is evident if the matrices \(A\) and \(B\) commute.
In this case \(d\sigma(\lambda)\) is an atomic measure supported on the spectrum of the matrix \(A\).
In general case, if the matrices \(A\) and \(B\) do not commute, the BMV conjecture remained an
open question for longer than 35 years. In 2011, Herbert
Stahl gave an affirmative answer to the BMV conjecture.\\[1.0ex]
\textbf{Theorem}\,(H.Stahl) \textit{Let \(A\) and \(B\) be \(n\times{}n\) hermitian matrices.}
\textit{Then the function \(f_{A,B}(t)\) defined by \eqref{TrF} is representable
as the bilateral Laplace transform \eqref{LaRm} of a non-negative measure \(d\sigma_{A,B}(\lambda)\)
supported on the closed interval \([\lambda_{\min},\lambda_{\max}]\).}

The first arXiv version of H.Stahl's Theorem appeared in \cite{S1}, the latest arXiv version -
in \cite{S2}, the journal publication - in \cite{S3}.
The proof of Herbert Stahl is based on ingenious considerations related to
Riemann surfaces of algebraic functions. In \cite{E1},\cite{E2} a simplified version of the Herbert Stahl proof is presented.

In the present paper we focus on the BMV conjecture for \(2\times 2\) matrices.
In this special case the BMV conjecture was proved in \cite[section 2]{MK}
using a perturbation series. We give a purely "matrix" proof of the BMV conjecture for
\(2\times2\) matrices.

\section{Exponentially convex functions.}

 \begin{defn}
 \label{decf}
A function \(f\) on \(\mathbb{R}\), \(f:\,\mathbb{R}\to[0,\infty)\), is said to be \emph{exponentially convex}
if
\begin{enumerate}
\item[\textup{1}.]
 For every nonnegative integer \(N\), for every choice of real numbers \(t_1\), \(t_2\),\(\,\ldots\,\), \(t_{N}\), and complex numbers
\(\xi_1\), \(\xi_2, \,\ldots\,, \xi_{N}\), the inequality holds
\begin{equation}
\label{pqf}
\sum\limits_{r,s=1}^{N}f(t_r+t_s)\xi_r\overline{\xi_s}\geq 0;
\end{equation}
\item[\textup{2}.]
The function \(f\) is continuous on \(\mathbb{R}\).
\end{enumerate}
\end{defn}
The class of exponentially convex functions was introduced by S.N.Bernstein, \cite{Ber1}, see \S 15 there.

From \eqref{pqf} it follows that
the inequality
\begin{math}
f(t_1+t_2)\leq\sqrt{f(2t_1)f(2t_2)}
\end{math}
holds for every \(t_1\in\mathbb{R},t_2\in\mathbb{R}\). Thus the alternative takes place: \\
\textit{If \(f\) is an exponentially convex function, then either \(f(t)\equiv 0\), or \(f(t)>0\) for every \(t\in\mathbb{R}\).}

\noindent
\begin{center}
\textbf{Properties of the class of exponentially convex functions.}
\end{center}
\begin{enumerate}
\item[\textup{P\,1}.] If \(f(t)\) is an exponentially convex function and \(c\geq0\) is a nonnegative constant, then the function \(cf(t)\) is exponentially convex.
\item[\textup{P\,2}.] If \(f_1(t)\) and \(f_2(t)\) are exponentially convex functions, then their sum
\(f_1(t)+f_2(t)\) is exponentially convex.
\item[\textup{P\,3}.] If \(f_1(t)\) and \(f_2(t)\) are exponentially convex functions, then their product
\(f_1(t)\cdot f_2(t)\) is exponentially convex.
\item[\textup{P\,4}.] If \(f(t)\) is an exponentially convex function and
\(a,\,b\) are real numbers, then the function \(f(at+b)\) is exponentially
convex.
\item[\textup{P\,5}.] Let \(\lbrace f_{n}(t)\rbrace_{1\leq n<\infty}\) be a sequence of exponentially
convex functions. We assume that for each \(t\in\mathbb{R}\) there exists the limit
\(f(t)=\lim_{n\to\infty}f_{n}(t)\), and that \(f(t)<\infty\ \forall t\in\mathbb{R}\).
Then the limiting function \(f(t)\) is exponentially convex.
\end{enumerate}

From the functional equation for the exponential function
it follows that for each real number \(\mu\), for every choice of real numbers \(t_1,t_2,\,\ldots\,\), \(t_{N}\) and complex numbers
\(\xi_1\), \(\xi_2, \,\ldots\,, \xi_{N}\), the equality holds
\begin{equation}
\label{ece}
\sum\limits_{r,s=1}^{N}e^{(t_r+t_s)\mu}\xi_r\overline{\xi_s}=
\bigg|\sum\limits_{p=0}^{N-1}e^{t_p\mu}\xi_p\,\bigg|^{\,2}\geq 0.
\end{equation}
The relation \eqref{ece} can be formulated as
\begin{lem}
\label{ECE}
For each real number \(\mu\), the function \(e^{t\mu}\) of the variable \(t\) is exponentially convex.
\end{lem}
For \(z\in\mathbb{C}\), the function \(\cosh z\), which is called \emph{the hyperbolic cosine of \(z\)}, is defined as
\begin{equation}
\label{dch}
\cosh z =\frac{1}{2}\big(e^z+e^{-z}\big).
\end{equation}
From Lemma \ref{ECE} and property P\,2 we obtain
\begin{lem}
\label{echc}
For each real \(\mu\), the function \(\cosh(t\,\mu)\) of the variable \(t\) is exponentially convex.
\end{lem}
The following result is well known.

\newpage

\begin{thm}[The representation theorem]\label{RepTe} {\ }\\[-2.0ex]
\begin{enumerate}
\item[\textup{1}.]
Let \(\sigma(d\mu)\) be a nonnegative measure on the real axis,
and let the function \(f(t)\) be defined as the two-sided Laplace transform of the measure
 \(\sigma(d\mu)\):
\begin{equation}
\label{rep}
f(t)=\int\limits_{\mu\in\mathbb{R}}e^{t\mu}\,\sigma(d\mu),
\end{equation}
where the integral in the right hand side of \eqref{rep} is finite for any \(t\in\mathbb{R}\). Then the function \(f\) is exponentially convex.
 \item[\textup{2}. ] Let \(f(t)\) be an exponentially convex function. Then this function \(f\) can be
  represented on \(\mathbb{R}\)  as a two-sided Laplace transform \eqref{rep} of a nonnegative measure \(\sigma(d\mu)\).  \textup{(}In particular, the integral in the right
     hand side of \eqref{rep} is finite for any \(t\in\mathbb{R}\).\textup{)} The representing measure \(\sigma(d\mu)\) is unique.
 \end{enumerate}
\end{thm}

The assertion 1 of the representation theorem is an evident consequence of Lemma~\,\ref{ECE}, of the properties P\,1, P\,2, P\,5, and of the definition of the integration operation.

 The proof of the assertion 2 can be found in \cite{A},\,Theorem 5.5.4, and in \cite{Wid},\,Theorem 21.

Of course, Lemma \ref{echc} is a special case of the representation theorem which corresponds to the representing measure
\[\sigma(d\nu)=1/2\big(\delta(\nu-\mu)+\delta(\nu+\mu)\big)\,d\nu,\] where \(\delta(\nu\mp\mu)\) are Dirak's \(\delta\)-functions supported at the points \(\pm\mu\).\\

Thus the Herbert Stahl theorem can be reformulated as follows:\\
\textit{Let \(A\) and \(B\) be Hermitian \(n\times{}n\) matrices.
Let the function \(f_{A,B}(t)\) is defined by \eqref{TrF}
for \(t\in(-\infty,\infty)\). Then the function \(f_{A,B}(t)\), considered as a function of the variable \(t\), is exponentially convex.}

\section{Reduction the BMV conjecture for general \(\boldsymbol{2\pmb{\times}2}\)
Hermitian matrices \(\boldsymbol{A}\) and \(\boldsymbol{B}\) to the case of special \(\boldsymbol{A}\) and \(\boldsymbol{B}\).}

\begin{lem}
\label{rt}
Let \(A\) and \(B\) be an arbitrary pair of \(2\times2\) Hermitian matrices.
Then there exists a pair \(A_0\), \(B_0\) of Hermitian \(2\times2\) matrices possessing the properties:
\begin{enumerate}
\item
The conditions
\begin{equation}
\label{ec} a).\tr A_0=0,\ \ b). \tr B_0=0,\ \ c). \tr A_0B_0=0.
\end{equation}
are satisfied.
\item
The trace-exponential functions \(f_{A,B}\) and \(f_{A_0,B_0}\) generated by these pairs
are related by the equality
\begin{equation}
\label{DuR}
f_{A,B}(t)=ce^{t\lambda}f_{A_0,B_0}(t+t_0),
\end{equation}
where
\(\lambda\) and \(t_0\) are some real numbers, \(c\) is a positive number.
\end{enumerate}
\end{lem}
\begin{rem}
From Lemma \ref{rt} it follows that in order to prove the BMV conjecture for arbitrary pair
\(A,B\) of Hermitian \(2\times2\) matrices, it is sufficient to prove this conjecture only for
pairs \(A_0,B_0\) satisfying the conditions \eqref{ec}.
\end{rem}
\begin{proof}[Proof of Lemma \ref{rt}] Let \(A\) and \(B\) be Hermitian matrices of size  \(2\times2\)
and \(I\) be the identity matrix of size \(2\times2\). Let us define
\begin{equation}
\label{A0}
A_0=A-\frac{\tr A}{2}I.
\end{equation}
Without loss of generality we can assume that
\begin{equation}
\label{A00}
A_0\not=0.
\end{equation}
Otherwise
\begin{equation*}
f_{A,B}(t)=ce^{t\lambda}, \ \ \textup{where} \ \ \lambda=\frac{\tr A}{2},\
c=\tr e^B>0,
\end{equation*}
and \eqref{DuR} holds with \(A_0=0,B_0=0\).  Since the matrix \(A_0\) is Hermitian, from \eqref{A00}
it follows that \(A_0^2\geq 0\), \(A_0^2\not=0\). Thus
\begin{equation}
\label{NZT}
\tr A_0^2>0.
\end{equation}
Let us define
\begin{equation}
\label{Dt0}
t_0=\frac{\tr A_0B}{\tr A_0^2}\ccomma
\end{equation}
\begin{equation}
B_0=B-\frac{\tr B}{2}I-t_0A_0.
\end{equation}
Since \(\tr I=2\) and \(\tr X\) depends on \(2\times2\) matrix \(X\) linearly, the conditions \(\tr A_0=0,\,\tr B_0=0\) are fulfilled. According to \eqref{Dt0},
the condition \(\tr A_0B_0=0\) is fulfilled as well.
Since
\begin{equation*}
A=A_0+\lambda I, \ B=B_0+\mu I +t_0A_0, \ \ \textup{where} \ \
 \lambda=\frac{\tr A}{2},\   \mu=\frac{\tr B}{2},
\end{equation*}
the linear matrix pencils \(tAt\) and \(tA_0+B_0\) are related by the equality
\begin{equation*}
tAt+B=(t\lambda+\mu)I+\big((t+t_0)A_0+B_0\big).
\end{equation*}
Therefore
\begin{math}
e^{tA+B}=e^{t\lambda+\mu}e^{(t+t_0)A_0+B_0}
\end{math},
that is the equality \eqref{DuR} holds with \(c=e^\mu\).
\end{proof}
\begin{lem}
\label{RCF}
Let \(A_0,\,B_0\) be Hermitian matrices of size \(2\times2\) satisfying the
condition \eqref{ec}, \(A_0\not=0\).
 Then there exists an unitary matrix \(U\) which reduces the matrices \(A_0,\,B_0\)
 to the form
 \begin{equation}
 \label{ReCF}
 UA_0U^{\ast}=\alpha\sigmal,\ \  UB_0U^{\ast}=\beta\taul,
 \end{equation}
 where  \(\alpha>0,\,\beta\geq0\) are numbers and
 \(\sigmal\), \(\taul\) are the Pauli matrices:
 \begin{equation}
 \label{CF}
\sigmal=
 \begin{bmatrix}
 1&\,\,\,0\\ 0&-1
 \end{bmatrix}\ccomma \  \qquad
 \taul=
 \begin{bmatrix}
0& 1\\ 1&0
 \end{bmatrix}\cdot
 \end{equation}
\end{lem}
\begin{proof}
Let \(U\) be an unitary matrix which reduces the Hermitian matrix \(A_0\) to
the diagonal form:
\begin{math}
UA_0U^{\ast}=
\bigl[\begin{smallmatrix}\lambda_1&0\\0&\lambda_2\end{smallmatrix}\bigr].
\end{math}
Since \(\tr A_0=0\), the equality \mbox{\(\lambda_1=-\lambda_2\)} holds. Since \(A_0\not=0\), also \(\lambda_1,\lambda_2\not=0\). Thus for some unitary matrix \(U\),
the first of the equalities \eqref{ReCF} holds with some number \(\alpha>0\). We fix this matrix \(U\) and \emph{define} the matrix
\begin{math}
\bigl[\begin{smallmatrix}b_{11}&b_{12}\\b_{21}&b_{22}\end{smallmatrix}\bigr]
=UB_0U^{\ast}.
\end{math}
Since \(\tr B_0=0\) and the matrix trace of is an unitary invariant, the equality
\(b_{11}+b_{22}=0\) holds. Since
\begin{math}
UA_0B_0U^{\ast}=UA_0U^{\ast}\cdot\, UB_0U^{\ast}=\bigl[\begin{smallmatrix}\alpha&0\\0&-\alpha\end{smallmatrix}\bigr]
\cdot\bigl[\begin{smallmatrix}b_{11}&b_{12}\\b_{21}&b_{22}\end{smallmatrix}\bigr]=
\bigl[\begin{smallmatrix}\alpha b_{11}&\alpha b_{12}\\-\alpha b_{21}&-\alpha b_{22}\end{smallmatrix}\bigr]
\end{math}
and \(\tr A_0B_0=0\), also
\begin{math}
\tr \bigl[\begin{smallmatrix}\alpha b_{11}&\alpha b_{12}\\-\alpha b_{21}&-\alpha b_{22}\end{smallmatrix}\bigr]=0
\end{math},
that is \(\alpha(b_{11}-b_{22})=0\). Since \(\alpha\not=0\), \(b_{11}-b_{22}=0\).          Finally, \(b_{11}=b_{22}=0\). Since the matrix
\begin{math}
\bigl[\begin{smallmatrix}b_{11}&b_{12}\\b_{21}&b_{22}\end{smallmatrix}\bigr]
\end{math}
is Hermitian, its entries \(b_{12}\) and \(b_{21}\) are conjugate complex numbers:
\(b_{12}=\overline{b_{21}}\).
The additional unitary equivalency transformation
\begin{math}
X\to
\Bigl[\begin{smallmatrix}
e^{i\vartheta}&0\\
0&1
\end{smallmatrix}\Bigr]
X
\Bigl[\begin{smallmatrix}
e^{-i\vartheta}&0\\
0&1
\end{smallmatrix}\Bigr]
\end{math}
does not change the matrix \(\sigmal\), but allows to reduce the matrix
\begin{math}
\bigl[\begin{smallmatrix}0&b_{12}\\\overline{b_{12}}&0\end{smallmatrix}\bigr]
\end{math}
 to the form \(\beta\taul\).
\end{proof}

\begin{lem}
 \label{PCF}
Let \(A_0\) and \(B_0\) be \(2\times2\) Hermitian matrices satisfying the
conditions \eqref{ec} \eqref{A00}, and \(U\) be the unitary matrix which reduces the pair \(A_0,\,B_0\) to the pair \(\alpha\sigmal,\,\beta\taul\)  according to \eqref{ReCF},\,\eqref{CF}. Then the trace-exponential functions generated by
the pairs \(A_0,\,B_0\) and \(\alpha\sigmal,\,\beta\taul\) coincide:
\begin{equation}
\label{ReEq}
 f_{A_0,B_0}(t)=f_{\alpha\sigmas,\,\beta\taus}(t).
\end{equation}
\end{lem}
\begin{proof}
\begin{multline*}
f_{A_0,B_0}(t)=\tr e^{tA_0+B_0}=\tr Ue^{tA_0+B_0}U^{\ast}=\\
=e^{U(tA_0+B_0)U^{\ast}}=e^{t\alpha\sigmas+\beta\taus}=
f_{\alpha\sigmas,\beta\taus}(t).
\end{multline*}
\end{proof}
\begin{rem}
\label{red}
From Lemmas \ref{rt}, \ref{RCF} and \ref{PCF} it follows that in order to prove the BMV conjecture for arbitrary pair \(A,B\)
of Hermitian \(2\times2\) matrices, it is enough to prove this conjecture for any pair of the form \(A=\alpha\sigmal,\,B=\beta\taul\) with
\(\alpha>0,\beta\geq0\).
\end{rem}

\section{The formulation of the main Theorem}
\begin{thm}[\textsf{The main theorem}]
\label{MaTe}
Let \(\alpha,\beta\) be arbitrary non-negative numbers and
\(\sigmal,\taul\) be the Pauli matrices
defined by \eqref{CF}.

Then the trace-exponential function \(f_{\alpha\sigmas,\beta\taus}(t)\)
generated by the pair of matrices \(\alpha\sigmal,\beta\taul\)
is exponentially convex.
\end{thm}

The trace-exponential function \(f_{\alpha\sigmas,\beta\taus}(t)\) can be easily
found explicitly:
\begin{equation}
\label{ETEF}
f_{\alpha\sigma,\beta\tau}(t)=2\cosh\sqrt{\alpha^2t^2+\beta^2}\,,
\end{equation}
where \(\cosh \zeta\) is the hyperbolic cosine function. However the exponential convexity of the function
\(\cosh\sqrt{\alpha^2t^2+\beta^2}\) is not evident.

There are different ways to prove the exponential convexity of the function \(f_{\alpha\sigma,\beta\tau}(t)\).
One can forget the "matrix" origin of the function  \(f_{\alpha\sigmas,\beta\taus}(t)\) and work with its
analytic expression \(\cosh\sqrt{\alpha^2t^2+\beta^2}\) only. The function \(\cosh\sqrt{\alpha^2t^2+\beta^2}\)
can be presented as a bilateral Laplace transform of some measure. The density of this measure can be expressed
in terms of the modified Bessel function \(I_1\). From this expression it is evident that the representing measure
is non-negative. However the calculation of the representing measure is not so transparent.

In the present paper we give a purely "matrix" proof of the BMV conjecture for
\(2\times2\) matrices. This proof is based on the Lie product formula
for the exponential of the sum of two matrices.
The proof also uses the commutation relations for the Pauli matrices and does not
use anything else.

\section{The proof of Theorem \(\boldsymbol{\ref{MaTe}}\).}\label{FiPr}

Since the trace-exponential function \(f_{\alpha\sigmas,\beta\taus}(t)\)
depends only on \(\beta^2\),
the equality
\begin{equation*}
f_{\alpha\sigmas,\beta\taus}(t)=f_{\alpha\sigmas,-\beta\taus}(t)
\end{equation*}
holds for any numbers \(\alpha,\beta\).
Therefore,
\begin{equation}
\label{AuMa}
f_{\alpha\sigmas,\beta\taus}(t)=\\
\tr\mathscr{E}(t;\alpha,\beta),
\end{equation}
where \(\mathscr{E}(t;\alpha,\beta)\) is the \(2\times2\) matrix-function:
\begin{equation}
\label{plmi}
\mathscr{E}(t;\alpha,\beta)=\frac{1}{2}\bigl[e^{t\alpha\sigmas+\beta\taus}
+e^{t\alpha\sigmas-\beta\taus}\bigr].
\end{equation}

 \begin{lem}[\textsf{A version of the Lie product formula}]
 \label{LPF}
 Let \(X\) and \(Y\) be square matrices of the same size, say \(n\times n\). Then
 \begin{equation}
 \label{lpf}
 e^{X+Y}=\lim_{N\to\infty}\Bigl(e^{\frac{X}{N}}\bigl(I+\tfrac{Y}{N}\bigr)\Bigr)^N.
 \end{equation}
 \end{lem}
 \begin{proof} Proof of the equality \eqref{lpf} can be modified from the
 proof which is presented in \cite[Theorem 2.10]{Ha}.
 \end{proof}
 \begin{proof}[Proof of Theorem \ref{MaTe}]
 We apply the equality \eqref{lpf} in the cases  \(X=t\alpha\sigmal\)
 and \(Y\) is one in two matrices \(Y=\beta\taul\),
 \(Y=-\beta\taul\).

The equality
\begin{equation}
\label{sq}
\taul^2=I
\end{equation}
and the commutation relation
 \begin{equation}
 \label{com}
 \taul\sigmal\taul=-\sigmal
\end{equation}
play crucial role in the proof of Theorem \ref{MaTe}.

 For every number \(\lambda\),
the matrix exponential \(e^{\lambda\sigmas}\) is
 a diagonal \(2\times 2\) matrix:
\begin{equation}
\label{DME}
e^{\lambda\sigmas}=
\begin{bmatrix}
e^{\lambda}&0\\
0&e^{-\lambda}
\end{bmatrix}.
\end{equation}
From \eqref{sq} and \eqref{com} the commutation relation
for the matrix exponentials \(e^{\lambda\sigmas}\), follows:
 \begin{gather}
  \label{ecom}
 \taul e^{\lambda\sigmas}\taul=e^{-\lambda\sigmas}, \ \ \forall\,\lambda\in\mathbb{R}.
 \end{gather}

 According to \eqref{plmi} and Lemma \ref{LPF},
 \begin{equation}
 \label{MeLR}
\mathscr{E}(t;\alpha,\beta)=\lim_{N\to\infty}\mathscr{E}_N(t;\alpha,\beta),
 \end{equation}
 where
 \begin{equation}
 \label{CiLPF}
 \mathscr{E}_N(t;\alpha,\beta)=
\frac{1}{2}\Bigl[ \mathscr{E}_N^{+}(t;\alpha,\beta)+
 \mathscr{E}_N^{-}(t;\alpha,\beta)\Bigr],
 \end{equation}
 \begin{equation}
 \label{P}
 \mathscr{E}_N^{+}(t;\alpha,\beta)=
\Bigl(e^{\frac{t\alpha\sigmas}{N}}\bigl(I+\tfrac{\beta\taup}{N}\bigr)\Bigr)^N,
\quad
 \mathscr{E}_N^{-}(t;\alpha,\beta)=
\Bigl(e^{\frac{t\alpha\sigmas}{N}}\bigl(I-\tfrac{\beta\taup}{N}\bigr)\Bigr)^N.
 \end{equation}
 From \eqref{P} it follows that

 \begin{subequations}
 \label{soae}
 \begin{gather}
 \label{soaep}
 \mathscr{E}_N^{+}(t;\alpha,\beta)=
 \sum\limits_{\varepsilon_1,\varepsilon_2\,\ldots\,\varepsilon_N}
 e^{\frac{t\alpha\sigmas}{N}}M_{\,\varepsilon_1}^{+}
 e^{\frac{t\alpha\sigmas}{N}}M_{\,\varepsilon_2}^{+}\,\ldots\,
 e^{\frac{t\alpha\sigmas}{N}}M_{\,\varepsilon_N}^{+},\\
 \label{soaem}
 \mathscr{E}_N^{-}(t;\alpha,\beta)=
 \sum\limits_{\varepsilon_1,\varepsilon_2\,\ldots\,\varepsilon_N}
 e^{\frac{t\alpha\sigmas}{N}}M_{\,\varepsilon_1}^{-}
 e^{\frac{t\alpha\sigmas}{N}}M_{\,\varepsilon_2}^{-}\,\ldots\,
 e^{\frac{t\alpha\sigmas}{N}}M_{\,\varepsilon_N}^{-},
\end{gather}
 \end{subequations}
 where each of \(\varepsilon_j,\,j=1,2,\,\ldots\,N\), takes value either \(0\),
 or \(1\),
 and the factors \(M_{\varepsilon}^{\pm}\) are:
 \begin{equation}
 \label{faM}
 M_{0}^{+}=I,\, M_{0}^{-}=I,\, M_{1}^{+}=\tfrac{\beta\taup}{N}\ccomma\,
 M_{1}^{-}=-\tfrac{\beta\taup}{N}\cdot
\end{equation}
  The sums in \eqref{soae} runs over all possible combinations
  \(\varepsilon_1,\varepsilon_2\,\ldots\,\varepsilon_N\) with either \(\varepsilon_j=0\)
  or  \(\varepsilon_j=1\). (There are \(2^N\) such combinations.)
  Grouping terms, we present the sums \eqref{soae} as iterated sums, where index summation \(m\) in the external sum runs over
  the set \(0,1,2,\,\ldots\,,N\). Each term in the internal sum is a product%
  \footnote{
  We omit the "trivial" factors \(M_{\,0}^{+}=I\),
  \(M_{\,0}^{-}=I\).
  }
   which contains \(N\) factors of the form \(e^{t\alpha\frac{1}{N}\sigmas}\) and \(m\) factors of the form \(\pm\frac{\beta\taup}{N}\). These factors in  general do not commute. So the generic term of the internal sum
  is the "word" \(W=F_1\cdot F_2\cdot\,\,\cdots\,\,\cdot F_k\cdot\,\,\cdots\,\,\cdot F_{N+m}\),
  consisting of two letters only: either \(F_k=e^{t\alpha\frac{1}{N}\sigmas}\) or \(F_k=\pm\frac{\beta\taup}{N}\). In the word \(W\), the letters
  \(F_k=\pm\frac{\beta\taup}{N}\)
   occupy \(m,\,0\leq m\leq N\), positions enumerated by \(k=p_1,k=p_2,\,\ldots\,,k=p_m\).
   Since each two neighbouring letters  of the form \(\pm\frac{\beta\taup}{N}\)
   must be separated by at least one letter of the form \(e^{t\alpha\frac{1}{N}\sigmas}\), the
    subscripts \(p_j,\,1\leq j\leq m\) enumerating positions of letters
    of the form \(\pm\frac{\taup}{N}\) must satisfy the conditions
    \begin{equation}
    \label{Adm}
    1<p_1,\,p_1+1<p_2,\,p_2+1<p_3,\,\ldots\,,p_{m-1}+1<p_{m},\,p_m\leq N+m.
    \end{equation}
    The letters
   \(F_k=e^{t\alpha\frac{1}{N}\sigmas}\) occupy the remaining \(N\)
   position.

    Thus
  \begin{subequations}
  \label{PoS}
  \begin{align}
  \label{PoSp}
  \mathscr{E}_N^{+}(t;\alpha,\beta)&=
  \sum\limits_{0\leq m\leq N}\Bigl(\frac{1}{N^m}\hspace*{-1.0ex}
  \sum\limits_{p_1,p_2,\,\ldots
  \,p_m}\hspace*{-2.0ex}W^{+}_{p_1,p_2,\,\ldots\,,p_m}\Bigr),\\
  \label{PoSm}
  \mathscr{E}_N^{-}(t;\alpha,\beta)&=
  \sum\limits_{0\leq m\leq N}\Bigl(\frac{1}{N^m}\hspace*{-1.0ex}
  \sum\limits_{p_1,p_2,\,\ldots
  \,p_m}\hspace*{-2.0ex}W^{-}_{p_1,p_2,\,\ldots\,,p_m}\Bigr),
  \end{align}
  \end{subequations}
  where
  \begin{subequations}
  \label{w}
  \begin{multline}
  \label{wp}
  W^{+}_{p_1,p_2,\,\ldots\,,p_m}=\beta^m\cdot
  e^{t\alpha\frac{p_1-1}{N}\sigmas}\cdot\taul\cdot
   e^{t\alpha\frac{p_2-p_1-1}{N}\sigmas}\cdot\taul\cdot
   e^{t\alpha\frac{p_3-p_2-1}{N}\sigmas}\cdot\taul\cdot\\
  \cdot e^{t\alpha\frac{p_4-p_3-1}{N}\sigmas} \cdot\taul\cdot\,\,\cdots\,\,\cdot
  e^{t\alpha\frac{p_m-p_{m-1}-1}{N}\sigmas}\cdot\taul\cdot
   e^{t\alpha(1-\frac{p_m-m}{N})\sigmas},
  \end{multline}
  \begin{multline}
   \label{wm}
  W^{-}_{p_1,p_2,\,\ldots\,,p_m}=(-\beta)^m\cdot
  e^{t\alpha\frac{p_1-1}{N}\sigmas}\cdot\taul\cdot
   e^{t\alpha\frac{p_2-p_1-1}{N}\sigmas}\cdot\taul\cdot
   e^{t\alpha\frac{p_3-p_2-1}{N}\sigmas}\cdot\taul\cdot\\
  \cdot e^{t\alpha\frac{p_4-p_3-1}{N}\sigmas} \cdot\taul\cdot\,\,\cdots\,\,\cdot
  e^{t\alpha\frac{p_m-p_{m-1}-1}{N}\sigmas}\cdot\taul\cdot
   e^{t\alpha(1-\frac{p_m-m}{N})\sigmas},
  \end{multline}
  \end{subequations}
  and the inner sums in  \eqref{PoS} runs over all sets of \(m\) integers
   \(p_1,p_2,\,\ldots\,p_m\) satisfying the conditions \eqref{Adm}.
    There are \(\binom{N}{m}=\frac{N!}{m!(N-m)!}\) such sets of \(m\) integers.

   By definition, the terms of the sums \eqref{PoS} corresponding to \(m=0\) are equal to \(e^{t\alpha\sigmas}\).

 In the expressions \eqref{PoS}, we should consider differently
 terms even \(m\) and with odd \(m\).

  If \(m\) is odd, then in the expressions \eqref{w}
  for the words \(W^{+}_{p_1,p_2,\,\ldots\,,p_m}\) and
  \(W^{-}_{p_1,p_2,\,\ldots\,,p_m}\),
 the factors \(\beta^m\) and \((-\beta)^m\) are of opposed signs.
 All other factors in these expressions coincide term by term. Therefore
 \begin{multline}
 \label{ann}
 W^{+}_{p_1,p_2,\,\ldots\,,p_m}+W^{-}_{p_1,p_2,\,\ldots\,,p_m}=0
 \quad \text{for each odd} \ m, \ \ \text{for each}\\
  \text{ set of subscripts  \(p_1,p_2,\,\ldots\,,p_m\) satisfying the conditions \eqref{Adm}}.
\end{multline}

 If \(m\) is even, then the factors \(\beta^m\) and \((-\beta)^m\)
 in the expressions \eqref{w}
  for the words \(W^{+}_{p_1,p_2,\,\ldots\,,p_m}\) and
  \(W^{-}_{p_1,p_2,\,\ldots\,,p_m}\)
  coincide.
  All other factors in these expressions coincide term by term as well.
  Therefore
 \begin{multline}
 \label{coi}
 W^{+}_{p_1,p_2,\,\ldots\,,p_m}=W^{-}_{p_1,p_2,\,\ldots\,,p_m}
 \quad \text{for each even} \ m, \ \ \text{for each}\\
  \text{ set of subscripts \(p_1,p_2,\,\ldots\,,p_m\) satisfying the conditions \eqref{Adm}}.
 \end{multline}
  For even \(m\), say \(m=2l\), the expression \eqref{w} for the word
  \( W^{+}_{p_1,p_2,\,\ldots\,,p_{2l}}=W^{-}_{p_1,p_2,\,\ldots\,,p_{2l}}\)
  can be simplified.  Let us choose and fix the set
  \(p_1,p_2,\,\) \(\ldots\,,p_{2l}\) of subscripts satisfying the conditions
  \eqref{Adm}.
 The factors \(\taul\)-s in the expression in \eqref{PoSp}-\eqref{PoSm} corresponding to this set of subscripts can be grouped by pairs of adjacent factors:
 \begin{multline}
 \label{AdF}
 W^{\pm}_{p_1,p_2,\,\ldots\,p_{2l}}=\beta^{2l}\cdot
 e^{t\alpha\frac{p_1-1}{N}\sigmas}\cdot(\taul
 e^{t\alpha\frac{p_2-p_1-1}{N}\sigmas}\taul)\cdot
 e^{t\alpha\frac{p_3-p_2-1}{N}\sigmas}\cdot
 \\
 \cdot\,\,\cdots\,\,\cdot
 e^{t\alpha\frac{p_{2l-1}-p_{2l-2}-1}{N}\sigmas}\cdot
 (\taul
 e^{t\alpha\frac{p_{2l}-p_{2l-1}-1}{N}\sigmas}\taul)\cdot
 e^{t\alpha(1-\frac{p_{2l}-2l}{N})\sigmas},
 \end{multline}
 Using \eqref{ecom}, we obtain that
 \begin{equation}
 \label{CoRE}
 \taul e^{t\alpha\frac{p_{2j}-p_{2j-1}-1}{N}\sigmas}\taul=
 e^{-t\alpha\frac{p_{2j}-p_{2j-1}-1}{N}\sigmas},\ \ 1\leq j\leq l.
 \end{equation}
  Hence
  \begin{equation}
  \label{expl}
   W^{+}_{p_1,p_2,\,\ldots\,,p_{2l}}=W^{-}_{p_1,p_2,\,\ldots\,,p_{2l}}=
   \beta^{2l} e^{t\alpha\mu_{p_1,\,\ldots\,p_{2l};N}\sigmas},
  \end{equation}
 where
 \begin{equation}
 \label{PoEx}
 \mu_{p_1,\,\ldots\,p_{2l};N}=\tfrac{1}{N}[2p_1-2p_2+2p_3-\,\cdots\,
 +2p_{2l-1}-2p_{2l}+N+2l].
 \end{equation}
 The numbers \(\mu_{p_1,\,\ldots\,p_{2l};N}\) satisfy the inequalities
 \begin{equation}
-(1-2/N) \leq\mu_{p_1,\,\ldots\,p_{2l};N}\leq1.
 \end{equation}
 From  \eqref{CiLPF}, \eqref{PoS}, \eqref{ann}, and \eqref{expl} it follows that
\begin{equation}
\label{FaEq}
\mathscr{E}_N(t;\alpha,\beta)=e^{t\alpha\sigmas}+
\sum\limits_{l:1\leq l\leq N/2 }\Bigl(\tfrac{\beta^{2l}}{N^{2l}}\hspace*{-2.0ex}
\sum\limits_{p_1,p_2,\,\ldots\,,p_{2l}}\hspace*{-2.0ex}
e^{t\alpha\mu_{p_1,p_2,\,\ldots\,,p_{2l};N}\sigmas}\Bigr),
\end{equation}
where \(p_1,p_2,\,\ldots\,,p_{2l}\) run over the set of integers satisfying the
conditions \eqref{Adm},
the numbers \(\mu_{p_1,\,\ldots\,p_{2l};N}\) are defined in \eqref{PoEx}.

The equality \eqref{FaEq} expresses the matrix function \(\mathscr{E}_N(t;\alpha,\beta)\) as a linear combination of the matrix
functions \(e^{t\alpha\mu\sigmas}\) with \emph{non-negative} coefficients,
which depend on \(\beta\):
\begin{equation}
\label{IRN}
\mathscr{E}_N(t;\alpha,\beta)=\int\limits_{\mu\in[-1,1]}
e^{t\alpha\mu\sigmas}\,\rho_N(d\mu),
\end{equation}
where
\begin{subequations}
\label{RemN}
\begin{gather}
\label{RemNa}
\rho_N(d\mu)=\sum\limits_{0\leq l\leq{}N/2}\rho_{N,l}(d\mu),\\
\label{RemNb}
\rho_{N,0}(d\mu)=\delta(\mu-1)\,d\mu, \ \
\rho_{N,l}(d\mu)=\tfrac{\beta^{2l}}{N^{2l}}\hspace*{-2.0ex}
\sum\limits_{p_1,p_2,\,\ldots\,,p_{2l}}\hspace*{-2.0ex}
\delta(\mu-\mu_{p_1,p_2,\,\ldots\,,p_{2l};N})\,d\mu,
\end{gather}
\(\delta(\mu)\) is the Dirak \(\delta\)-function,
the summation in \eqref{RemNb} runs over all sets of integers \(p_1,p_2,\,\ldots\,,p_{2l}\) satisfying the conditions \eqref{Adm} with \(m=2l\),
the numbers \(\mu_{p_1,p_2,\,\ldots\,,p_{2l};N}\) are the same that in
\eqref{PoEx}.
\end{subequations}
In view of \eqref{DME}, the matrix-function \(\mathscr{E}_N(t;\alpha,\beta)\)
is diagonal:
 \begin{equation}
\label{diN}
\mathscr{E}_{N}(t;\alpha,\beta)=
\begin{bmatrix}
e_{1,N}(t;\alpha,\beta)&0\\
0&e_{2,N}(t;\alpha,\beta)
\end{bmatrix}.
\end{equation}
The diagonal entries \(e_{1,N}(t;\alpha,\beta),\,e_{2,N}(t;\alpha,\beta)\)
are representable as
\begin{equation}
\label{RDEN}
e_{1,N}(t;\alpha,\beta)=\hspace*{-1.5ex}
\int\limits_{\mu\in[-1,1]}\hspace*{-1.5ex}e^{t\alpha\mu}\rho_N(d\mu),
\ \ e_{2,N}(t;\alpha,\beta)=\hspace*{-1.5ex}
\int\limits_{\mu\in[-1,1]}\hspace*{-1.5ex}e^{-t\alpha\mu}\rho_N(d\mu).
\end{equation}
According to Theorem \ref{RepTe}, each of the functions \(e_{1,N}(t;\alpha,\beta),\,e_{2,N}(t;\alpha,\beta)\)
is exponentially convex. Their sum, which is the trace of the matrix
\(\mathscr{E}_{N}(t;\alpha,\beta)\), is exponentially convex.
In view of \eqref{MeLR}, the function \(\tr \mathscr{E}(t;\alpha,\beta)\)
is exponentially convex. The reference to \eqref{AuMa} completes the proof
of Theorem~\ref{MaTe}.
\end{proof}
\begin{rem}
\label{LUB}
For each \(\beta\geq0\), the family of the measures
\(\bigl\{\rho_N(d\mu)\bigr\}_N\) is uniformly
bounded with respect to \(N\):
\begin{equation}
\label{ube}
\int\limits_{\mu\in[-1,1]}\rho_N(d\mu)\leq e^{\beta}.
\end{equation}
Indeed, for each \(N\), the cardinality of the set of integers \(p_1,p_2,\,\ldots\,,p_{m}\) satisfying the
conditions \eqref{Adm} is equal to \(\binom{N}{m}=\frac{N!}{m!(N-m)!}\cdot\) According to
\eqref{RemNb},
\begin{equation*}
\int\limits_{\mu\in[-1,1]}\rho_{N,l}(d\mu)=\binom{N}{2l}\frac{\beta^{2l}}{N^{2l}},\ \
\forall l:\,0\leq 2l\leq N.
\end{equation*}
Taking into account \eqref{RemNa}, we obtain
\begin{equation*}
\int\limits_{\mu\in[-1,1]}\hspace*{-1.0ex}\rho_{N}(d\mu)=\hspace*{-1.5ex}
\sum\limits_{l:0\leq 2l\leq N}\binom{N}{2l}\tfrac{\beta^{2l}}{N^{2l}}<\hspace*{-1.0ex}
\sum\limits_{0\leq k\leq N}\binom{N}{k}\biggl(\frac{\beta}{N}\biggr)^k\!
\!=\Bigl(1+\tfrac{\beta}{N}\Bigr)^N<e^{\beta}.
\end{equation*}
\end{rem}
\section{A Theorem on the integral representation of a \(2\times2\) matrix function}
\begin{thm}
\label{TIR}
Let \(\beta\) be a non-negative, \(\sigmal\) and \(\taul\) be the Pauli
matrices which were defined in \eqref{CF}.
 For each \(\beta\geq0\), let \(\mathscr{E}(t;\beta)\) be the
matrix function of the variable \(t\in\mathbb{R}\) which is defined by the equality
\begin{equation}
\label{dE}
\mathscr{E}(t;\beta)=\frac{e^{t\sigmas+\beta\taus}+e^{t\sigmas-\beta\taus}}{2}\cdot
\end{equation}
(The value \(\beta\) is considered as a parameter.)

Then there exists a non-negative \emph{scalar} measure \(\rho(d\mu)\) supported on the interval
\([-1,1]\)  such that the integral representation
\begin{equation}
\label{IR}
\mathscr{E}(t;\beta)=\int\limits_{\mu\in[-1,1]}e^{t\mu\sigmas}\rho(d\mu), \ \ \forall t\in\mathbb{R}.
\end{equation}
holds. The  measure \(\rho\) admits the estimate
\begin{equation}
\label{Er}
\int\limits_{\mu\in[-1,1]}\rho(d\mu)\leq e^{\beta}.
\end{equation}
\end{thm}
\begin{proof}
We start from the integral representation \eqref{IRN}, where we can set \(\alpha=1\). The inequality
\eqref{ube} means that for each \(\beta\), the family of measures \(\bigl\{\rho_N\bigr\}\)
is bounded with respect to \(N\). Therefore the family of measures \(\bigl\{\rho_N\bigr\}\)
is weakly compact. From \eqref{IRN} and \eqref{MeLR} it follows that representation \eqref{IR}
holds with every measure \(\rho\) which is a weak limiting point of the family \(\bigl\{\rho_N\bigr\}\).
Actually such \(\rho\) is unique.
\end{proof}
\begin{rem}
\label{MFE}
The measure \(\rho(d\mu)\) which appears in the integral representation \eqref{IR} can be presented explicitly. The matrix-function \(\mathscr{E}(t;\beta)\) is diagonal:
\begin{equation}
\label{DM}
\mathscr{E}(t;\beta)=
\begin{bmatrix}
e_{1}(t;\beta)&0\\
0&e_2(t;\beta)
\end{bmatrix}\cdot
\end{equation}
From \eqref{dE} we find that
\begin{subequations}
\label{EE}
\begin{align}
\label{EEa}
e_{1}(t;\beta)&=\cosh\sqrt{t^2+\beta^2}+t\cdot\frac{\sinh\sqrt{t^2+\beta^2}}{\sqrt{t^2+\beta^2}}\ccomma\\
\label{EEb}
e_{2}(t;\beta)&=\cosh\sqrt{t^2+\beta^2}-t\cdot\frac{\sinh\sqrt{t^2+\beta^2}}{\sqrt{t^2+\beta^2}}\cdot
\end{align}
\end{subequations}
The function \(\cosh\sqrt{t^2+\beta^2}\) admits the integral representation
\begin{equation}
\label{ira}
\cosh\sqrt{t^2+\beta^2}=\cosh t+\int\limits_{\mu\in[-1,1]}\widehat{d}(\mu,\beta)e^{\mu t}\,d\mu,
\end{equation}
where
\begin{equation}
\label{ILT}
\widehat{d}(\mu,\beta)=\frac{\beta}{2\sqrt{1-\mu^2}}I_1(\beta\sqrt{1-\mu^2}),
\ \ -1\leq\mu\leq1.
\end{equation}
\(I_1(\,.\,)\) is the modified Bessel function. The appropriate calculation can be found in
\cite[Section 3]{Ka}, in particular Lemma 3.2 there.
From \eqref{EE}, \eqref{ira} and \eqref{ILT} we obtain the following expression for
 the measure \(d\rho(\mu)\) from \eqref{IR}:
 \begin{equation}
 \label{fem}
 \rho(d\mu)=\delta(\mu-1)\,d\mu+(1+\mu)\widehat{d}(\mu,\beta)\,d\mu.
 \end{equation}
\end{rem}

\end{document}